\documentclass[11pt,reqno]{amsart}

\usepackage{amscd}
\usepackage{amsfonts}
\usepackage{amsmath}
\usepackage{amssymb}
\usepackage{amsthm}
\usepackage{fancyhdr}
\usepackage{latexsym}
\usepackage[colorlinks=true, pdfstartview=FitV, linkcolor=blue,
            citecolor=blue, urlcolor=blue]{hyperref}

\usepackage{cancel}

\input epsf
\input texdraw
\input txdtools.tex
\input xy
\xyoption{all}

\usepackage{caption}
\usepackage{subcaption}

\usepackage{cases}
\usepackage{cleveref}

\usepackage{xurl}

\usepackage{enumerate}

\usepackage[ruled,vlined]{algorithm2e}

\SetAlCapSkip{1em}

\usepackage{mathrsfs}

\usepackage{mathtools}

\usepackage{microtype}

\usepackage{color}

\definecolor{Red}{rgb}{1.00, 0.00, 0.00}
\definecolor{DarkGreen}{rgb}{0.00, 1.00, 0.00}
\definecolor{Blue}{rgb}{0.00, 0.00, 1.00}
\definecolor{Cyan}{rgb}{0.00, 1.00, 1.00}
\definecolor{Magenta}{rgb}{1.00, 0.00, 1.00}
\definecolor{DeepSkyBlue}{rgb}{0.00, 0.75, 1.00}
\definecolor{DarkGreen}{rgb}{0.00, 0.39, 0.00}
\definecolor{SpringGreen}{rgb}{0.00, 1.00, 0.50}
\definecolor{DarkOrange}{rgb}{1.00, 0.55, 0.00}
\definecolor{OrangeRed}{rgb}{1.00, 0.27, 0.00}
\definecolor{DeepPink}{rgb}{1.00, 0.08, 0.57}
\definecolor{DarkViolet}{rgb}{0.58, 0.00, 0.82}
\definecolor{SaddleBrown}{rgb}{0.54, 0.27, 0.07}
\definecolor{Black}{rgb}{0.00, 0.00, 0.00}
\definecolor{dark-magenta}{rgb}{.5,0,.5}
\definecolor{myblack}{rgb}{0,0,0}
\definecolor{darkgray}{gray}{0.5}
\definecolor{lightgray}{gray}{0.75}

\usepackage[pdftex]{graphicx}
\usepackage{epstopdf}

\newcommand{\p}{\partial}
\newcommand{\zz}{\mathbb{Z}}

\newcommand{\nn}{\mathbb{N}}
\newcommand{\qq}{\mathbb{Q}}

\theoremstyle{plain}  
\newtheorem{theorem}{Theorem}
\newtheorem{proposition}{Proposition}
\newtheorem{lemma}{Lemma}

\theoremstyle{definition}
\newtheorem{definition}{Definition}

\theoremstyle{definition}
\newtheorem{example}[theorem]{Example}

\usepackage{tikz}

\usetikzlibrary{patterns}

\begin{document}

\title{
  Computing the Bernstein Polynomial and 
        the Krull-type Dimension of finitely generated $\boldsymbol{D}$-modules
}
  
\author{Harry Prieto}


\begin{abstract}
We establish the existence of the Bernstein polynomial in one indeterminate $t$, and provide a method for its explicit computation. The Bernstein polynomial is associated with finitely generated modules over the Weyl algebra, known as $D$-modules, and is notoriously difficult to compute directly. Our approach is constructive, offering a systematic method to compute the Bernstein polynomial and its associated invariants explicitly. We begin by introducing the Weyl algebra as a ring of operators and stating some of its main properties, followed by considering the class of numerical polynomials. We  then develop a generalization of the theory of Gröbner bases specifically for $D$-modules and use it to compute the Bernstein polynomial and its invariants. As an application of the properties of the Bernstein polynomial, we develop the concept of the Krull-type dimension for $D$-modules, which sheds light on the structure of these modules.
\end{abstract}

\date{Nov 7, 2024}

\maketitle

\markboth{Bernstein Polynomials}
{H. Prieto}

\section{Introduction}
\label{sec:1}
 \setcounter{equation}{0}

The Bernstein polynomial is a numerical polynomial, analogous to the Hilbert polynomial in commutative algebra, that accounts for the dimension of a finitely generated module over the Weyl algebra. The principal aim of this paper is to provide a constructive approach to the existence of the Bernstein polynomial, developing a methodology to compute it and its invariants, as will be seen in the results of Section \ref{sec:4}. The existence of the Bernstein polynomial is usually established nonconstructively and, as noted in \cite[p.~95]{cou_95}, its explicit computation is notoriously complicated. The approach we take here relies on a generalization of the classical theory of Gr\"{o}bner bases and on results about numerical polynomials \cite{kondra_99} to obtain a relatively straightforward methodology for the computation of the Bernstein polynomial. Similar approaches have been successfully implemented in other contexts, for example, by D\"{o}nch and Levin \cite{levin_13} for bivariate dimension polynomials, and by Levin \cite{levin_07} for multivariable dimension polynomials with respect to several orderings, multivariate Berstein-type polynomials \cite{levin_24.1}, and bivariate difference dimension polynomials \cite{levin_24.2}.

In Section \ref{sec:2}, we begin by introducing the Weyl algebra as a ring of operators and stating some of its basic properties. Section \ref{sec:3} introduces the class of numerical polynomials $f(t)$, which become integer-valued for sufficiently large values of the indeterminate $t$. We also state two important facts about numerical polynomials that will assist in computing the Bernstein polynomial. Section \ref{sec:4} is devoted to developing a generalization of the classical theory of  Gr\"{o}bner bases for the case of finitely generated modules over the Weyl algebra $A_n$. There, we will first introduce a way of ordering the elements of the Weyl algebra, followed by a concept of ``divisibility." These notions are then extended to the case of a finitely generated free module over $A_n$. Next we define a concept of reduction and present an analogue of the division algorithm for the Weyl algebra. We then define the concept of a Gr\"{o}bner basis for $D$-modules, state some of its properties and provide a criterion to find one (Theorem \ref{buch}, Buchbereger's algorithm). Section \ref{sec:5} contains the main results of the paper, namely, establishing the existence of and computing the Bernstein polynomial of a finitely generated $D$-module (Theorem \ref{bern_basis} and Theorem \ref{bern_exist}) using both the theory of  Gr\"{o}bner bases and the results about numerical polynomials from previous sections. We also use these ideas to exhibit the invariants of the Bernstein polynomial (Theorem \ref{invar}), namely, its degree and multiplicity. This section features fully worked examples of explicitly computing the Bernstein polynomial of $D$-modules, thus illustrating the main result. In Section \ref{sec:6}, we use the Bernstein polynomial and some of its properties to define the concept of the Krull-type dimension for $D$-modules. These concepts present a further application of the ideas of Section \ref{sec:5}. The paper concludes with a result about the Krull-type  dimension of a finitely generated $D$-module (Theorem \ref{Krull-type}).

\section{Preliminaries }
\label{sec:2}
 \setcounter{equation}{0}

Throughout this paper, every ring homomorphism is unitary (maps identity onto identity), and every subring of a ring contains the identity of the ring. By a module over a ring $R$ we mean a unitary left $R$-module. In what follows, we consider a Weyl algebra as an algebra of differential operators over a polynomial ring. More precisely, let $K$ be a field and $K[X]$ the ring of polynomials $K[x_1, \ldots , x_n]$ in $n$ commuting indeterminates with coefficients in $K$ ($n \ge 1)$.
We denote by $End_K (K[X])$ the $K$-algebra of linear operators (mappings) $K[X] \rightarrow K[X]$. This is a noncommutative ring with identity, whose operations are addition and composition of operators.

Let $f \in K[X]$ be a polynomial. Multiplication by $x_1, \ldots, x_n $ gives $K$-linear operators on $K[X]$ denoted by $\hat{x}_1, \ldots, \hat{x}_n $, respectively, acting on $f$ by the formula $\hat{x}_i (f) =x_i \cdot f$, for $i = 1, \ldots, n$. For simplicity of notation, we will write $x_i$ for both the operator $\hat{x}_i$ and the variable $x_i$. On $K[X]$ we also have pairwise commuting $K$-linear derivations $\p_1, \ldots , \p_n$ that are elements of $End_K (K[X])$ defined by $\p_i (f) = \p f / \p x_i$, for $i = 1, \ldots ,n$.
With this notation, we make the following definition.

\begin{definition}
The $\boldsymbol{n}$\textbf{th Weyl algebra}, denoted by $A_n(K)$, is the $K$-subalgebra of $End_K (K[X])$ generated by the operators
$x_1,\ldots , x_n$ 
and 
$\p_1, \ldots ,\p_n$. 
\end{definition} 

We will adopt the convention that $A_0(K) = K$. Furthermore, we will often write $A_n$ instead of $A_n(K)$ if the field over which the algebra is defined is clear from the context, or no emphasis is needed.
Although commutativity holds within each operator ($x_i x_j = x_j x_i$ and $\p_i \p_j = \p_j \p_i$ for any $i,i \in \{ 1, \ldots , n \}$, and $x_i \p_j = \p_j x_i$ whenever $i \neq j$), the Weyl algebra is not commutative. Indeed, for any $f \in K[X]$, we have 
\begin{align*}
    \p_i x_i (f) = x_i \frac{ \p f} {\p x_i} + f,
\end{align*}
obtained by applying the product rule for differentiation ---also called the Leibniz rule. Thus, 
$\p_i \cdot x_i = x_i \cdot \p_i + 1$, where $1$ is the identity operator.

The elements of $A_n$ are linear combinations, with coefficients in $K$, of the products of the generators $x_1,\ldots , x_n$ 
and 
$\p_1, \ldots ,\p_n$. 
To describe these elements, we use a multi-index notation. \textbf{Multi-indices} $\alpha$ and $\beta$ of size $n$ are elements of $\nn^n$, that is, $n$-tuples of nonnegative integers $\alpha = (\alpha_1, \ldots , \alpha_n)$ and $\beta = (\beta_1, \ldots , \beta_n)$.
In this notation, we set $x^{\alpha} = x_{1}^{\alpha_1} \cdots x_{n}^{\alpha_n}$
and 
$\p^{\beta} = \p_{1}^{\beta_1} \cdots \p_{n}^{\beta_n}$. 
We also define the \textbf{degrees} of the power products $x^\alpha$ and $\p^\beta$ as the numbers
$|\alpha | = \alpha_1 + \cdots + \alpha_n$ and 
$|\beta | = \beta_1 + \cdots + \beta_n$, respectively. 
A power product of the form $\theta = x^{\alpha} \p^\beta \in A_n$, with $\alpha, \beta \in \nn^n$, is called a \textbf{monomial} in $A_n$. 
The monomials $x^\alpha$ and $\p^\beta$ will be denoted by $\theta_x$ and $\theta_\p$, respectively.
The \textbf{degree} of such a monomial $\theta$ is defined as $|\alpha | + |\beta | $, and it is denoted by $deg (\theta)$.
Furthermore, we set the factorial of a multi-index $\alpha \in \nn^n$ as $\alpha! = \alpha_1! \cdots \alpha_n!$.
In what follows, $\Theta$ will denote the set of all monomials  $x^{\alpha} \p^{\beta}$ of $A_n$.
Moreover, for any $r \in \nn$, $\Theta (r)$ will denote the set $\{\theta \in \Theta  : deg(\theta) \le r \}$. 

The following proposition says that when considered as a vector space over $K$, the Weyl algebra has a particular basis. 
\begin{proposition}[{\cite[Proposition~1.2]{bio}}]
\label{basis}
Each element in $A_n$ can be written in a unique way as a finite sum
$\sum k_{\alpha, \beta} x^{\alpha} \p^\beta $, where the coefficients $k_{\alpha, \beta} \in K$. 
In other words, the set of monomials
$\mathcal{B} = \{ x^{\alpha} \p^\beta :\alpha, \beta \in \nn^n \} $ 
is a basis of $A_n$ as a $K$-vector space.
\end{proposition}

The Weyl algebra can be considered as a filtered ring by means of the following filtration on $A_n$ known as the \textbf{Bernstein filtration}, which will be defined using the degree of the elements of $A_n$.
For $r \ge 0$, let $B_r$ denote the $K$-vector space of all elements of $A_n$ generated by $\{ x^\alpha \p^\beta : |\alpha| + |\beta | \le r \}$. These are finite-dimensional $K$-vector subspaces of $A_n$: 
$B_0 =K$; $B_1 = K +K x_1 + \cdots + K x_n + K \p_1 + \cdots + K \p_n $, and so on.

\section{Numerical Polynomials in one indeterminate}
\label{sec:3}
 \setcounter{equation}{0}

\begin{definition}
A polynomial $f(t)$ in one indeterminate $t$  with rational coefficients is called \textbf{numerical} if $f(t) \in \zz$ for all sufficiently large $t \in \nn$, that is, if there exists $r \in \nn$ such that $f(s) \in \zz$ for any $s \in \nn$ with $r \le s$.
\end{definition}

Clearly, every polynomial with integer coefficients is numerical. For a general example of a numerical polynomial with non-integer coefficients, we can consider the polynomial in one indeterminate $t$
\begin{align*}
    &\binom{t}{n} = \frac{t(t-1) \cdots (t-n+1)}{n!}
    \quad (n \ge 1);
    &\binom{t}{0}= 1.
\end{align*}

The following proposition, proved in \cite[Proposition~2.1.3]{kondra_99}, gives a canonical representation of a numerical polynomial.

\begin{proposition}
\label{canon}
Let $f(t) \in \qq[t]$ be a polynomial of degree $m$ in one indeterminate $t$. Then $f(t)$ is numerical if and only if it can be represented in the form
\begin{align}
\label{int_canon}
    f(t) =
            \sum^m_{i =0} a_i \binom{t+i}{i},
\end{align}
where the coefficients $a_0, \ldots , a_m$ are integers uniquely determined by $f(t)$. 
\end{proposition}

For the remaining of this section we will deal with subsets of $\nn^m$, where $m$ is a positive integer. If $A \subseteq \nn^m$, $V_A$ will denote the set of all $m$-tuples $v = (v_1, \ldots, v_m) \in \nn^m$ that are no greater than or equal to any point of $A$ with respect to the product order on $\nn^m$. The \textit{product order} on $\nn^m$ is a partial order $\le_P$ such that $(c_1, \ldots, c_m) \le_P (c'_1, \ldots, c'_m)$ if and only if $c_i \le c'_i$, for all $i =1, \ldots, m$. Clearly, an element $v = (v_1, \ldots, v_m) \in \nn^m$ belongs to $V_A$ if and only if for any element $(a_1, \ldots, a_m) \in A$ there exist $i \in \nn$, with $1 \le i \le m$, such that $a_i > v_i$.
Also, we let $\nn_p = \{1, \ldots, p \}$ $(p \in \zz, p \ge 1)$ be the set of the first $p$ integers. 
Furthermore, for any $r \in \nn$, we will denote by $A(r)$ the subset of $A$ consisting of all $m$-tuples $(a_1, \ldots, a_m)$ such that $a_1 + \cdots + a_m \le r$.
\medskip

The next two propositions, proved respectively in \cite[Chapter~0,~Lemma~16]{kolch} and \cite[Proposition~2.2.11]{kondra_99}, introduce certain numerical polynomias associated with subsets of $\nn^m$, and give explicit formulas for their computation.

\begin{proposition}
\label{kolchin_dim}
Let $A \subseteq \nn^m$. There exists a numerical polynomial $\omega_A(t)$ with the following properties:
\begin{enumerate}[{\em (i)}]
    \item 
    $\omega_A(s) = Card\; V_A (s)$ for all sufficiently large $s \in \nn$ (as usual, $Card~V$ denotes the number of elements of a finite set $V$).
    \item
    $deg(\omega_A) \le m$; the equality occurs if and only if the set $A$ is empty, in which case we have
\begin{align*}
    \omega_A(t) = \binom{t + m}{m}.
\end{align*}
\item
$\omega_A (t) = 0$ if and only if $(0, \ldots, 0) \in A$.
\end{enumerate}
\end{proposition}

\begin{definition}
The polynomial $\omega_A (t)$, whose existence is established in the preceding proposition, is called the \textbf{Kolchin dimension polynomial} of the set $A \subseteq \nn^m$.
\end{definition}

\textbf{Remark.} In \cite[Chapter~0,~Lemma~16]{kolch} it is shown that any infinite subset of $\nn^m$ contains a sequence that is strictly increasing with respect to the product order. It follows that if $A \subseteq \nn^m$, then the set $A^*$ of all minimal points of $A$ with respect to the product order is finite. It is easy to see that $\omega_A (t) = \omega_{A^*} (t)$. The following proposition, proved in \cite[Chapter~II]{kondra_99}, provides a way to compute the Kolchin polynomial of any finite subset of $\nn^m$, and therefore, of any subset of $\nn^m$.

\begin{proposition}
\label{num_compu}
Let $A = \{a_1, \ldots , a_p \}$ be a finite subset of $\nn^m$, where $m$ is a fixed positive integer, and let $a_i = (a_{i1}, \ldots, a_{im})$ for $i = 1, \ldots, p$. Furthermore, for any $\ell \in \zz$, $0 \le \ell \le p$, let $\Lambda(\ell, p)$ denote the set of all $\ell$-element subets of $N_p$, and for any $\lambda \in \Lambda(\ell,p)$ let $\bar{a}_{\lambda j} = max\{a_{ij} : i \in \lambda \}$ for $1 \le j \le m$ (in other words, if $\lambda = \{i_1, \ldots, i_l\}$, then $\bar{a}_{\lambda j}$ denotes the greatest $jth$ coordinate of the elements $a_{i_1}, \ldots, a_{i_l}$). Also, let $b_\lambda = \sum_{j=1}^m \bar{a}_{\lambda j}$. Then
\begin{align*}
    \omega_A(t) =
    \sum_{\ell =0}^p (-1)^\ell 
        \sum_{\lambda \in \Lambda(\ell,p)} \binom{t +m - b_\lambda}{m}.
\end{align*}
\end{proposition}

\section{Gröbner Bases for \texorpdfstring{$D$}{}-modules}
\label{sec:4}
 \setcounter{equation}{0}

In this section, we adapt the classical technique of Gröbner bases to $D$-modules or modules over the Weyl alegbra $A_n$. The theory we develop here will be used in the next section to prove the existence of the Bernstein polynomial and compute its invariants. Analogous methodologies have been implemented by A. Levin in \cite{levin_07}, \cite{levin_24.1}, \cite{levin_24.2}, and by Dönch and Levin in \cite{levin_13} for the computation of a dimension polynomial in two variables, showing that such a polynomial carries more invariants than the Bernstein dimension polynomial. The approach we take in what follows draws on some of the ideas presented in those works. Complete presentations of the classical theory can be found in \cite{adams} and \cite{beck} (which includes applications to the computation of Hilbert polynomials of graded and filtered modules over polynomial rings). Generalizations of the notion of Gröbner bases to rings of differential operators are carried out in \cite{insa} and \cite[Chapter~4]{kondra_99}.

We will begin by defining the concept of an order for the monomials of the Weyl Algebra, followed by a definition of divisibility. Using these ideas, we will then define the concept of reduction, which in turn will lead to an analog of the division algorithm for $D$-modules over $A_n$. Then we define the concept of a Gröbner basis in this context, giving equivalent conditions for it. Finally, we discuss how to compute Gröbner bases, developing the concept of $S$-polynomials and a version of Buchberger's algorithm for that purpose. Once this is accomplished, we will use this method to compute the Bernstein polynomial and its invariants in the next section.

Let $\Theta$ denote the set of all monomials in the Weyl algebra $A_n$. We can define a natural order on the monomials of $\Theta$ with respect to degree as follows.

\begin{definition}
An \textbf{order} on the set of all monomials $\Theta$ of $A_n$ is a total order $<$ on $\Theta$, such that
$ \p_n  < \cdots  < \p_1 < x_n < \cdots < x_1$, 
and if 
$\theta = x^{\alpha} \p^{\beta} = x_1^{\alpha_1} \cdots x_n^{\alpha_n} \p_1^{\beta_1} \cdots \p_n^{\beta_n}$ 
and 
$\theta' = x^{\gamma} \p^{\delta} = x_1^{\gamma_1} \cdots x_n^{\gamma_n} \p_1^{\delta_1} \cdots \p_n^{\delta_n}$ 
are two elements of $\Theta$, then
\begin{align*}
\theta < \theta'
 \quad
\Longleftrightarrow
    \quad(deg(\theta), \alpha_1, \ldots, \alpha_n, \beta_1, \ldots, \beta_n)  
    <_{lex}
    (deg(\theta'), \gamma_1, \ldots, \gamma_n, \delta_1, \ldots, \delta_n),
\end{align*}
where $<_{lex}$ denotes the lexicographic order on $\nn^{2n+1}$. 
\end{definition}

Note that we are using the symbol ``$<$'' in two different ways, both for the order relation on the multi-indices in $\nn^n$ and for the order introduced on the monomials of $\Theta$. The context will make clear which meaning is intended. 

Also note that since the set $\nn^{2n+1}$ is well-ordered with respect to the lexicographic order, our ordering of $\Theta$ turns it into a well-ordered set.
Additionally, if $\theta < \theta'$, we also write $\theta' > \theta$. In what follows, the monomial order on $\Theta$ introduced above is fixed.

\begin{definition}
Let $D = \sum k_{\alpha ,\beta} x^{\alpha} \p^{\beta}$ be a nonzero element of  $A_n$ and let $<$ be a monomial order on $\Theta$.
\begin{enumerate}[(i)]
    \item 
    The \textbf{leading monomial} of $D$, denoted by LM$(D)$, is the greatest monomial $x^\alpha \p^\beta$ that appears in $D$ with a nonzero coefficient. 
    \item
    The coefficient of the leading monomial of $D$ is called the \textbf{leading coefficient} of $D$ and is denoted by LC$(D)$. 
    \item
    The \textbf{leading term} of $D$ is defined as LT$(D) = $ LC$(D) \cdot$ LM $(D)$.
\end{enumerate}
\end{definition}

Next, we define a notion of divisibility for the elements of $A_n$ as follows. 
Let 
$ \theta = x^{\alpha} \p^{\beta} = x_{1}^{\alpha_1} \ldots x_{n}^{\alpha_n} \p_{1}^{\beta_1} \ldots \p_{n}^{\beta_n}$,
and
$\theta' = x^{\gamma} \p^{\delta} = x_{1}^{\gamma_1} \ldots x_{n}^{\gamma_n} \p_{1}^{\delta_1} \ldots \p_{n}^{\delta_n}$ 
be two elements of $\Theta$. 
We say that $\theta$ \textbf{divides} $\theta'$ if $x^\alpha$ divides $x^\gamma$ and $\p^\beta$ divides $\p^\delta$, that is, 
$\alpha_i \leq \gamma_i$ and $\beta_i \leq \delta_i$ for $i = 1, \ldots ,n$. In this case we also say that $\theta'$ is a \textbf{multiple} of $\theta$ and write $\theta | \theta'$.

According to this definition, if $\theta | \theta'$, then there exist elements $\theta_0, \theta_1, \ldots \theta_k \in \Theta$ such that 
$\theta' = \theta_{0} \theta - \sum_{i=1}^{k} \theta_i$,
where the  
$deg (\theta_{0}) + deg(\theta)  = deg (\theta')$, 
and
$deg (\theta_i) < deg (\theta')$
for $i= 1, \ldots ,k$. The monomial $\theta_0$ is denoted by
    $
    \frac{\theta'}{\theta}$.

To illustrate, if $n =1$, then $\theta = x\p^2$ divides $\theta' = x^{2}\p^3$, and we can write $\theta' = \theta_{0} \theta - \theta_1$, where $\theta_0 = x\p$, and $\theta_1 = x\p^2$.

By the \textbf{least common multiple} of two elements $\theta', \theta'' \in \Theta$ we mean the element 
lcm$(\theta', \theta'') = $ lcm$(\theta'_x, \theta''_x) \cdot$lcm$(\theta'_\p, \theta''_\p)$.
Thus, if 
$\theta =$ lcm$(\theta', \theta'')$,
then both
$\theta' | \theta$, $\theta'' | \theta$, 
and whenever $\theta' | \tau $ and $\theta'' | \tau $, for some $\tau \in \Theta$, 
we have $\theta | \tau$.

Let $E$ be a free $A_n(K)$-module with a finite set of free generators $\{e_1, \ldots ,e_m \}$. Then $E$ can be regarded as a $K$-vector space with basis 
$\Theta e = \{ \theta e_i : \theta \in \Theta, 1 \le i \le m \}$ 
whose elements are called \textbf{monomials}.

Since the set of all monomials $\Theta e$ is a basis of the $K$-vector space $E$, every nonzero element $f \in E$ has a unique representation of the form
\begin{align}
\label{f_repre}
    f 
    =
    a_{1} \theta_{1} e_{i_1} + \cdots + a_m \theta_{m} e_{i_m},
\end{align}
where $\theta_j \in \Theta$, $0 \neq a_j \in K$ $(1 \le j \le m)$ and the monomials 
$\theta_1 e_{i_1}, \ldots   ,\theta_m e_{i_m}$
are all distinct. We say that a monomial $u$ \textit{appears in f} (or that \textit{f contains $u$}) if $u$ is one of the monomials $\theta_k e_{i_k}$ in the representation  (\ref{f_repre}), that is, the coefficient of $u$ in $f$ is not zero. 
The greatest monomial that appears in $f$ with a nonzero coefficient is called the \textit{leading monomial} of $f$ and is denoted by $LM(f)$. The coefficient of the leading monomial of $f$ is called the \textit{leading coefficient} of $f$ and is denoted by $LC(f)$. The \textit{leading term} of $f$ is defined as $LT(f) =  LC(f) \cdot LM (f)$.

We say that a monomial $u = \theta' e_i$ is a \textbf{multiple} of a monomial $v = \theta e_j \quad (\theta, \theta' \in \Theta, 1 \le i,j \le m)$ if $i = j$ and $\theta | \theta'$. In this case we also say that $v$ divides $u$, write $v | u$, and define
\begin{align*}
    \frac{u}{v}
    \coloneqq 
    \frac{\theta'}{\theta}.
\end{align*}

If $c_1 u$ and $c_2 v$ are terms, their ratio is defined as $\frac{c_1 u}{c_2 v} = (c_1 c_2^{-1}) \frac{\theta'}{\theta}$. Then we say that the term $c_2 v$ \textit{divides} $c_1 u$.

The \textit{least common multiple} of two monomials $w_1 = \theta_1 e_i$ and $w_2 = \theta_2 e_j $ is defined as 
\begin{align*}
    \text{lcm}(w_1,w_2) 
    \coloneqq 
    \begin{cases}
    \text{lcm}(\theta_1, \theta_2) e_i,  &\text{if } i = j, \\
    0,                              &\text{if } i \neq j.
    \end{cases}
\end{align*}
We will consider a monomial order on the set $\Theta e$ defined as follows. 
For monomials $\theta e_i = x^\alpha \p^\beta e_i$ and
 $\theta' e_j = x^\delta \p^\gamma e_j$ of $\Theta e$, we say that
 \begin{align*}
    \theta e_i 
    <
    \theta' e_j 
    \Longleftrightarrow
\begin{cases}
&\theta < \theta' \\
&\text{or} \\
&\theta = \theta' \; \; \text{and} \;\; i<j.
\end{cases}
\end{align*}

\begin{example}
Consider a free $A_2$-module $E$ with two free generators $e_1$ and $e_2$. 
Let $f = -2 x^3_1 \p_1 e_1 - \p^3_2 + 5x_2 e_1 + 3x_1 \p^2_1 e_2 + 4x_2 e_2 +\p^2_2e_2 \in E$. 
Arranging the monomials of $f$ in decreasing order, we get
\begin{align*}
    f = -2 x^3_1 \p_1  e_1  + 3x_1 \p^2_1 e_2 - \p^3_2e_1 +\p^2_2 e_2 + 4x_2 e_2 + 5x_2e_1,
\end{align*}
and so $LM(f) = x^3_1 \p_1  e_1 $, $LC(f) = -2$, and $LT(f) = -2 x^3_1 \p_1 e_1 $.
\end{example}

\begin{definition}
Let $f, g \in E$, with $g \neq 0$. We say that $f$ is \textbf{reduced} with respect to $g$ if $f$ does not contain any multiple of $LM (g)$.
\end{definition}

\begin{definition}
Given $f, g, h \in E$, with $g \neq 0$, we say that the element $f$ reduces to $h$ modulo $g$ in one step, written
\begin{align*}
    f 
    \xrightarrow {g}
    h,
\end{align*}
if and only if $f$ contains some term $w$ with coefficient $a \neq 0$ such that LM$(g) | w$, and
\begin{align*}
    &h 
    =
    f - \frac{a w}{LT(g)}g.
\end{align*}
\end{definition}
\begin{example}
Let $f = \p^3_1 e_1 + 2x^3_1e_1 + 3x^2_1 \p_2 e_2 + 3x_2 \p^2_1 e_2 + \p^2_1 e_2 +4x_2 e_2$
and
$g = x_1 e_1 + e_1 + \p^2_1 e_2$ elements of the free $A_2$-module $E$ with free generators $e_1$ and $e_2$.
Then $LM(g) = \p^2_1 e_2$, and we see that the terms  $x^2_1 \p_2 e_2 $ and $\p^2_1 e_2$ of $f$ are divisible by $LM(g)$. Selecting the first of these terms, we can perform reduction modulo $g$, where $h_1 = f - \frac{3x_2 \p^2_1 e_2}{\p^2_1 e_2} g$, to obtain 
\begin{align*}
     f 
    \xrightarrow {g} h_1 =
    \p^3_1 e_1+ 2x^3_1 e_1 -3x_1 x_2 e_1 -3x_2 e_1 + 3x^2_1\p_2 e_2 + \p^2_1 e_2 + 4x_2e_2.
\end{align*}
After this step, $h_1$ still contains a term divisible by $LM(g)$, namely, $\p^2_1 e_2$. So we perform a second reduction step to obtain
\begin{align*}
     h_1
     \xrightarrow {g} h_2 =
    \p^3_1 e_1+ 2x^3_1 e_1 -3x_1 x_2 e_1 -3x_2 e_1 - x_1 e_1 -e_1 + 3x^2_1\p_2 e_2 + 4x_2e_2.
\end{align*}
And since $h_2$ does not contain any multiple of $LM(g)$, $h_2$ is reduced with respect to $g$.
\end{example}
\begin{definition}
\label{redu_s}
Let $f, h \in E$, and let $G = \{g_1, \ldots ,g_s\}$ be a finite set of nonzero elements of $E$.  
We say that $f$ reduces to $h$ modulo $G$, denoted
\begin{align*}
    f 
    \xrightarrow{G}
    h,
\end{align*}
if and only if there exist elements 
$g^{(1)}, g^{(2)}, \ldots ,g^{(t)} \in G$ 
and
$h^{(1)}, h^{(2)}, \ldots ,h^{(t-1)} \in E$ 
such that
\begin{align*}
    f 
    \xrightarrow{g^{(1)}}
    h^{(1)}
    \xrightarrow{g^{(2)}}
    h^{(2)}
    \cdots
    \xrightarrow{g^{(t-1)}}
    h^{(t-1)}
    \xrightarrow{g^{(t)}}
    h.
\end{align*}
\end{definition}
\begin{example}
Let $g_1 = x^2 \p e_1 - \p e_1 + x^2 e_2$ and $g_2 = x e_1 + \p^2 e_2 -x e_2 \in E$, where $E$ is a free $A_1$-module with free generators $e_1$ and $e_2$. Let $G =\{ g_1, g_2 \}$ and $f = \p^2 e_1 +2x^2 \p e_1 + \p^2 e_2$. Then
\begin{align*}
     f 
    \xrightarrow{G}
   \p^2 e_1 - x e_1 + 2 \p e_1 - 2x^2 e_2 + x e_2.
\end{align*}
To obtain this, we successively eliminate any multiples of $LM(g_1)$ and $LM(g_2)$ that appear in $f$. Accordingly, in the reduction process shown below, we first compute $ h_1 = f - \frac{2x^2 \p e_1}{x^2 \p e_1} g_1$, followed by $h = h_1 - \frac {\p^2 e_2}{\p^2 e_2}g_2$ in the next step:
\begin{align*}
    f = \p^2 e_1 +2x^2 \p e_1 + \p^2 e_2
     &\xrightarrow{g_1}
        \p^2 e_1 + 2\p e_1 -2x^2 e_2 + \p^2 e_2  \\
     &\xrightarrow{g_2} 
        \p^2 e_1 - x e_1 + 2 \p e_1 - 2x^2 e_2 + x e_2.
\end{align*}
Notice that $h = \p^2 e_1 - x e_1 + 2 \p e_1 - 2x^2 e_2 + x e_2$ cannot be reduced further by $g_1$ or $g_2$. This is because no term of $h$ is divisible by $LM(g_1) = x^2 \p e_1$ or by $LM(g_2) = \p^2 e_2$. 
\end{example}
\begin{definition}
An element $r \in E$ is said to be \textit{reduced with respect} to a set 
$G = \{g_1, \ldots ,g_s\}$
if $r = 0$ or no monomial that appears in $r$ is divisible by one of the LM$(g_i),~i = 1, \ldots s$.
If  
$ 
    f 
    \xrightarrow{G}
    r 
$
and $r$ is reduced with respect to $G$, then we call $r$ a remainder for $f$ with respect to $G$.
\end{definition}
The reduction process described in the preceding definitions allows us to state a version of the division algorithm analogous to the division algorithm for polynomials. The specific process and its justification is encapsulated in the following theorem and its proof, while the algorithm is given as Algorithm \ref{algo_1} below.

\begin{theorem}
\label{division}
With the notation above, let $f \in E$, and let $G = \{g_1, \ldots ,g_s\}$ be a finite subset of $E$.
Then there exist elements 
$r \in E$ and $q_1, \ldots ,q_s \in A_n(K)$ 
such that
\begin{align*}
    f 
    =
    q_{1}g_{1} + \cdots + q_{s}g_{s} +r,
\end{align*}
where $r$ is reduced with respect to $G$, and 
LM$(f) =$ max$\{\text{LM}(q_1 g_1), \ldots ,\text{LM}(q_s g_s), \text{LM}(r) \} $.

\end{theorem}
\begin{proof}
If $f$ is reduced with respect to $G$, then the statement is obvious, because we can set $r =f$. Suppose that $f$ is not reduced with respect to $G$.
In what follows, a term $w_h$ will be called a $G$-\textit{leader} of an nonzero element $h \in E$ if $w_h$ is the greatest among the monomials $w$ satisfying the following properties:
\begin{enumerate}[(i)]
    \item 
    $w$ appears in $h$;
    \item
    $w$ is a multiple of some $LM(g_i)$, for $1 \le i \le s$.
\end{enumerate}
Let $w_f$ be the $G$-leader of an element $f \in E$, and let $c_f$ be the coefficient of $w_f$ in the representation (\ref{f_repre}) of $f$. Then $LM(g_i) | w_f$, for some $i$, $1 \le i \le s$. Without loss of generality, we may assume that $i$ corresponds to the greatest leading monomial $g_i$ satisfying the above conditions. 
Let $\frac{w_f}{LM(g_i)} = \eta$, so that we can write
$f = \eta LT(g_i) +S_i$, where $S_i$ denotes an element of $E$ whose $G$-leader is lower than $w_f$.

Let 
\begin{align*}
    f' &= f - c_f (LC(g_i))^{-1} \eta g_i.
\end{align*}
Clearly, then, $f'$ does not contain $w_f$. Furthermore, $f'$ cannot contain any term $w'$ such that $w_f < w'$, with $LM(g_i) | w'$. Indeed, by the choice of $w_f$ such a term $w'$
cannot appear in $f$; it cannot appear in $\eta g_i$  either because $LM(\eta g_i) = w_f < w'$. Thus, $w'$ cannot appear in $f' = f - c_f (LC(g_i))^{-1} \eta g_i $, and so the $G$-leader of $f'$ is strictly less than $w_f$. Applying the same procedure to $f'$ and continuing this way, we obtain an element $r \in E$ such that $f-r \in A_n g_1 + \cdots A_n g_s$, and $r$ is reduced with respect to $G$.
\end{proof}
\begin{algorithm}[H]
\label{algo_1}
   \SetAlgoNoLine
  \KwIn{$f, g_1, \ldots, g_s \in E$ with $g_i \neq 0 \; (1 \le i \le s)$}
  \KwOut{An element $r \in E$ such that there exist $q_1, \cdots , q_s \in A_n$
  with $f = q_{1}g_{1} + \cdots + q_{s}g_{s} +r,$  and $r$ is reduced with respect to $G$}
  \textbf{Initialization:} 
  $q_1 \coloneqq 0, \cdots , q_s \coloneqq 0, r \coloneqq 0, g \coloneqq f$ \\
\While{$g \neq 0$ and there exist $i \in \{1, \cdots, s\}$, and a monomial $w$ that appears in $g_i$ with nonzero coefficient $c$ such that LM$(g_i) |w$}{
\smallskip
    $z \coloneqq$ the greatest monomial $w$ satisfying the above conditions;  \\
    $k \coloneqq$ the least number $i$ for which $LM(g_i) | z$. \\
    $q_k \coloneqq q_k + \frac{cz}{LT(g_k)}g_k$; \quad $g \coloneqq g - \frac{cz}{LT(g_k)}g_k$.
    }
    \caption{Reduction Algorithm}
\end{algorithm}
\medskip

We are now ready to define Gröbner bases for $D$-modules.

\begin{definition}
Let $N$ be a submodule of a free $A_n$-module $E$ with basis $\{ e_1, \ldots , e_m \}$.
A finite subset of nonzero elements 
$G = \{g_1, \ldots ,g_s\}$
contained in $N$ is called a \textbf{Gröbner basis} for $N$ if for any nonzero element $f \in N$, there exist $g_i \in G$ such that 
LM$(g_i)$ divides LM$(f)$. We say that the set $G$ is a  Gröbner basis provided $G$ is a  Gröbner basis for the $A_n$-submodule 
$N = A_n g_1 + \cdots + A_n g_s$ 
of $E$ it generates.
\end{definition}

Next we introduce the analog of the concept of $S$-polynomials for $D$-modules. This construction is the main ingredient in finding a Gröbner basis from a set of generators of an $A_n$-module.
\begin{definition}
\label{s_poly}
Let $f$ and $g$ be two elements in the free $A_n$-module $E$. Let $L = lcm ($LM$(f), $LM$(g))$.
The element
\begin{align*}
    S(f,g)
    =
    \frac{L}{LT(f)}f
    -
     \frac{L}{LT(g)}g
\end{align*}
is called the \textbf{S-polynomial} of $f$ and $g$.
\end{definition}

\begin{example}
Let $E$ be a free $A_2$-module with free generators $e_1$ and $e_2$. Consider the elements $f,g \in E$, given by
$f = x_1^2 e_1 + 5x_1 \p_1^2 e_2 +x_2 e_2$, and $g = x_1^2 \p_1 e_1 + 3x_1^3 \p_1 e_2 + \p_1 e_2 $. Then
$LM(f) = x_1 \p^2_1 e_2$, $LM(g) = x^3_1 \p_1 e_2$ and $lcm (LM(f), LM(g)) = x^3_1 \p^2_1 e_2$. Thus
\begin{align*}
S(f, g) &=
    \frac{x_1^3 \p_1^2 e_2}{5x_1 \p_1^2 e_2}f 
      - \frac{x_1^3 \p_1^2 e_2}{3x_1^3 \p_1 e_2}g = 
    \frac{x_1^2}{5}f - \frac{\p_1}{3}g  \\
    &= \frac{1}{5}x_1^4 e_1 - \frac{1}{3} x_1^2 \p_1^2 e_1 - \frac{2}{3}x_1 \p_1 e_1 +\frac{1}{5}x_1^2 x_2 e_2 
    - 3x_1^2 \p_1 e_2 -\frac{1}{3}\p^2_1 e_2.
\end{align*}
Note that the leading monomial of each of the elements $(x_1^2 /5)f$ and $(\p_1 /3)g$ is $x_1^3 \p_1^2 e_2$ while the leading monomial of the $S$-polynomial is $x_1^4 e_1 < x_1^3 \p_1^2 e_2$.
\end{example}

\begin{lemma}
\label{combo}
Let $f_1, \ldots, f_s \in E$ $(s \ge 1)$ be such that $LM(f_i) =\delta > 0$ for all $i = 1, \ldots, s$. Let $f = \sum_{i=1}^s c_i \theta_i f_i$, where $\theta_i \in \Theta$, and $c_i \in K$ $(1 \le i \le s)$. If $LM(f) < \delta$, then $f$ is a linear combination, with coefficients in $K$, of $S(f_i,f_j)$, $1 \le i \le s$.
\end{lemma}
\begin{proof}
First we write $f_i = a_i\delta +$ lower terms, with $a_i \in K$. Then by hypothesis 
$\sum_{i =1}^s c_i a_i = 0$, 
since the $c_i$'s are in $K$. And by definition, 
$S(f_i, f_j) = \frac{1}{a_i} f_i - \frac{1}{a_j} f_j$, 
since $LM(f_i) = LM(f_j) =\delta$. Thus,
\begin{align*}
    f 
    &=
    c_1 f_1 + \cdots + c_s f_s.  \\
    &=
    c_1 a_1 \left( \frac{1}{a_1} f_1 \right) + \cdots + c_s a_s \left( \frac{1}{a_s} f_s \right)  \\
    &=
    c_1 a_1 \left( \frac{1}{a_1} f_1 - \frac{1}{a_2} f_2 \right) + (c_1 a_1 + c_2 a_2)  \left( \frac{1}{a_2} f_2 - \frac{1}{a_3} f_3 \right) + \cdots  \\
    &\quad\;+
    (c_1 a_1 + \cdots + c_{s-1} a_{s-1}) \left( \frac{1}{a_{s-1}} f_{s-1} - \frac{1}{a_s} f_s \right)
    + (c_1 a_1 + \cdots + c_s a_s) \frac{1}{a_s} f_s  \\
    &=
    c_1 a_1 S(f_1, f_2) + (c_1 a_1 + c_2 a_2) S(f_2, f_3) + \cdots + (c_1 a_1 + \cdots + c_{s-1} a_{s-1})  S(f_{s-1}, f_s),
\end{align*}
since $(c_1 a_1 + \cdots + c_s a_s) = 0$.
\end{proof}
With the preceding definition and lemma we are now ready to state and prove the theoretical foundation of the algorithm for computing Gröbner bases. 
\begin{theorem}[Buchberger's Criterion]
\label{buch}
Let $G = \{g_1, \ldots, g_s \}$ be a set of nonzero elements of an $A_n$-submodule of $E$ with respect to an order $<$. Then $G$ is a Gröbner basis for the submodule $N = \langle g_1, \ldots, g_s \rangle$ of $E$ if and only if for all $i \neq j$,
\begin{align*}
    S(g_i, g_j)
    \xrightarrow{G}
    0.
\end{align*}
\end{theorem}
\begin{proof}
First suppose that $G  = \{g_1, \ldots, g_s \}$ is a Gröbner basis for $N = \langle g_1, \ldots, g_s \rangle$. Let $f \in E$. Then, by Theorem \ref{division}, there exists $r \in E$, reduced with respect to $G$, such that  $f 
    \xrightarrow{G}
    r$.
Thus $f - r \in N$ and so $f \in N$ if and only if $r \in N$. Clearly, if $r =0$ (that is,  $f 
    \xrightarrow{G}
    0$),
then $f \in N$. Conversely, if $f \in N$ and $r \neq 0$ then $r \in N$, and there exists $i \in \{ 1, \ldots, s\}$ such that LM$(g_i)$ divides LM$(r)$. This contradicts the fact that $r$ is reduced with respect to $G$. Thus, $r =0$ and  $f 
    \xrightarrow{G}
    0$.

To prove the implication in the other direction,
suppose that $S(g_i, g_j)  \xrightarrow{G} 0$ for all $i \neq j$. We use the fact that $G$ is a Gröbner basis for $N$ if an only if for all $f \in N$, there exist 
    $h_1, \ldots , h_s \in A_n$ such that 
    $f = h_1 g_1 + \cdots + h_s g_s$
    and LM$(f) = \max_{1 \le i \le t} \{ \text{LM}(h_i g_i) \}$. So let $f \in N$. From among all the representations of $f$ as a linear combination of the $g_i$, we can choose to write $f = \sum_{i=1}^s h_i g_i$, where  
\begin{align*}
    \delta = \max_{1 \le i \le s} \{LM(h_i) LM(g_i) \}
\end{align*}
is minimal, by the well-ordering property of our term order. If $\delta = LM(f)$, then $LM(f)$ is divisible by $LM(g_i)$ for some $i$, and the result follows.

It remains to consider the case when $LM(f) < \delta$. First we isolate the terms whose degree equals $\delta$. To that end, let $\mathcal{I} = \{i : LM(h_i)LM(g_i) = \delta \}$. For $i \in \mathcal{I}$, we write $h_i =c_i\delta+$ lower terms. Set $g = \sum_{i \in S} c_i \delta_i g_i$. Then, $LM(\delta_i g_i) = \delta$, for all $i \in \mathcal{I}$, but $LM(g) < \delta$. By Lemma \ref{combo}, we can write $g$ as a linear combination with coefficients $d_{ij} \in K$, that is,
\begin{align*}
    g 
    =
    \sum_{i,j \in \mathcal{I}, i\neq j}
    d_{ij} S(\delta_i g_i, \delta_j g_j).
\end{align*}
Since $\delta = lcm(LM(\delta_i, g_i), LM(\delta_j,g_j))$, then
\begin{align*}
    S(\delta_i g_i, \delta_j g_j)
    &=
     \frac{\delta}{LT(\delta_i g_i)}\delta_i g_i
    -
      \frac{\delta}{LT(\delta_j g_j)}\delta_j g_j  \\
    &=
     \frac{\delta}{LT(g_i)} g_i
    -
      \frac{\delta}{LT(g_j)} g_j = \frac{\delta}{\delta_{ij}}S(g_i, g_j),
\end{align*}
where $\delta_{ij} = lcm (LM(g_i), LM(g_j))$. By hypothesis, $S(g_i, g_j) \xrightarrow{G} 0 $, and thus $S(\delta_i g_i, \delta_j g_j) \xrightarrow{G} 0$. We can now write the $S$-polynomial as 
\begin{align*}
    S(\delta_i g_i, \delta_j g_j)
    =
    \sum_{\ell = 1}^{t} h_{ij\ell} g_\ell ,
\end{align*}
with the property that (by Theorem \ref{division})
\begin{align*}
\max_{1 \le \ell \le s} \{LM(h_{ij\ell} LM(g_\ell) \}
&=
LM (  S(\delta_i g_i, \delta_j g_j ) ) \\
&<
\max \{LM(\delta_i g_i), LM (\delta_j g_j) \}
=
\delta.
\end{align*}
Substituting these expressions into $g$ above allows us to write $f$ as a polynomial combination of the $g_i$, namely, $f = \sum_{i = 1}^{s} h'_{i} g_i$, where $\max_{1 \le i \le s} \{LM(h'_{i}) LM(g_i) \} < \delta$. This contradicts the minimality of $\delta$ and completes the proof.
\end{proof} 
The last theorem allows us to construct a Gröbner basis of an $A_n$-submodule $E$ starting with the usual Gröbner basis of $N$ (with respect to the monomial order $<$). We now give the generalization of Buchberger's algorithm to the Weyl algebra, whose correctness is justified in the following theorem and presented as Algorithm \ref{algo_2} below.
\begin{theorem}
Let $F = \{f_1, \ldots , f_s \}$ $(f_i \neq 0, \; 1 \le i \le s)$, be a subset of a free $A_n$-module $E$. Then Buchberger's algorithm (Algorithm \ref{algo_2}) produces a Gröbner basis for the $A_n$-submodule $N = \langle f_1, \ldots , f_s \rangle$ of $E$.
\end{theorem}
\begin{proof}
First we show that the algorithm terminates. Suppose, for contradiction, that the algorithm does not terminate. Then, we can construct a set $G_i$ strictly larger that $G_{i-1}$ and obtain a strictly increasing infinite sequence
\begin{align*}
    G_1 \subsetneq  G_2 \subsetneq   G_3 \subsetneq \cdots .
\end{align*}
Each $G_i$ is obtained from $G_{i-1}$ by adding some $h \in N$ to $G_{i-1}$, where $h$ is the nonzero reduction, with respect to $G_{i-1}$ of an $S$-polynomial of two elements of $G_{i-1}$. Since $h$ is reduced with respect to $G_{i-1}$, we have that $LT(h) \notin LT(G_{i-1})$. This gives
\begin{align*}
     LT(G_1) \subsetneq  LT(G_2) \subsetneq   LT(G_3) \subsetneq \cdots .
\end{align*}
This is a strictly ascending chain of $A_n$-submodules, contradicting the fact that $A_n$ is a left and right Noetherian ring {\cite[Proposition 2.8]{bio}}. Thus the algorithm terminates.

We have $F \subseteq G \subseteq N$, and hence $N = \langle f_1, \ldots , f_s \rangle \subseteq \langle g_1, \ldots , g_s \rangle \subseteq N$. Thus, $G$ is a generating set for the submodule $N$. Also, if $g_i , g_j$ are elements of $G$,  then $S(g_i, g_j) \xrightarrow{G} 0$ by construction. Therefore, $G$ is a Gröbner basis for the $A_n$-submodule $N$ by Theorem \ref{buch}.
\end{proof}
\medskip

\begin{algorithm}[H]
\label{algo_2}
   \SetAlgoNoLine
  \KwIn{$F = \{f_1, \ldots, f_s \} \subseteq E$ with $f_i \neq 0 \; (1 \le i \le s)$}
  \KwOut{$G = \{g_1, \cdots , g_s \}$,
  a Gröbner basis for $(f_1, \ldots, f_s)$}
  \textbf{Initialization:} 
  $G \coloneqq F$, $\tilde{G} \coloneqq \{ \{f_i, f_j \} : f_i \neq f_j \in G \}$ \\
  \While{$ \tilde{G} \neq \emptyset$ }{
   Choose any $\{f_i, f_j \} \in \tilde{G}$ \\
   $\tilde{G} \coloneqq \tilde{G} - \{ \{f_i, f_j \} \}$ \\
   $S(f_i, f_j) \xrightarrow{G} h$, where $h$ is reduced with respect to $G$ \\
    \If{ $h \neq 0$}{
    $\tilde{G} \coloneqq \tilde{G} \cup \{ \{u, h \} :$ for all $u \in G \}$ \\
    $G \coloneqq G \cup \{h \}$
    }
    }
    \caption{Buchberger's Algorithm for $D$-modules}
\end{algorithm}
\medskip

\begin{example}
\label{grob_ex}
Let $E$ be a free $A_2$-module with free generators $e_1$ and $e_2$. Consider the $A_n$-submodule $N$ of $E$ generated by
\begin{align*}
    g_1 &= x_1^2 \p_1^3 e_1 + \p_1^5 e_1, \\
    g_2 &= x_2^2 e_1 - x_1 e_2.
\end{align*}
To find a Gr\"{o}bner basis for $N$, we start by setting $G = \{g_1, g_2 \}$. Since $LM(g_1) = x_1^2 \p_1^3 e_1$ and $LM(g_2) = x_2^2 e_1$, then $lcm(LM(g_1), LM(g_2)) = x^2_1 x^2_2 \p^3_1 e_1$, and using Definition \ref{s_poly} we have
\begin{align*}
    S(g_1, g_2)= x^2_2 g_1 - x^2_1 \p^3_1 g_2 =
    x_2^2 \p_1^5 e_1 + x^3_1 \p_1^3 e_2 + 3x_1^2 \p_1^2 e_2 .
\end{align*}
Since this polynomial contains a multiple of $LM(g_2)$, it can be reduced module $g_2$ in one step by computing $h = S(g_1, g_2) - ( x_2^2 \p_1^5 e_1 / x_2^2 e_1 ) g_2$ as follows:
\begin{align*}
    S(g_1, g_2)= x_2^2 \p_1^5 e_1 + x^3_1 \p_1^3 e_2 + 3x_1^2 \p_1^2 e_2  \;
    \xrightarrow {g_2}  \;
    x_1^3\p_1^3 e_2 + 3x_1^2 \p_1^2 e_2 + x_1 \p_1^5 e_2 +5\p_1^4 e_2.
\end{align*}
It can be seen now that this last element is reduced with respect to $G$, so we set $g_3 = x_1^3\p_1^3 e_2 + 3x_1^2 \p_1^2 e_2 + x_1 \p_1^5 e_2 +5\p_1^4 e_2$, and add it to $G$. 
Next, we have $S(g_1, g_3) = S(g_2, g_3) =0$. Therefore, by Theorem \ref{buch},
$G =\{g_1, g_2, g_3\}$ is a Gr\"{o}bner basis of $N$.

\end{example}

\section{Berstein polynomials of \texorpdfstring{$D$}{}-modules and their invariants}
\label{sec:5}
 \setcounter{equation}{0}

This section provides a constructive proof of the existence of the Bernstein polynomial, which accounts for the dimension of a finitely generated $D$-module. This polynomial was introduced by I. Bernstein in \cite{bern_71} as an analog of the classical Hilbert polynomial (see \cite[Section~1.9]{eis}) for a finitely generated filtered module over a Weyl algebra. Among the many applications of this concept, it allowed Bernstein to prove Gelfand's conjecture about the  meromorphic extensions of the function the complex gamma function. The proof of the existence of the Bernstein polynomial we present here will apply the theory of Gr\"{o}bner bases we developed in the previous section, along with the properties of numerical polynomials stated in Section \ref{sec:2}.

The first result we present is the main step in computing the Bernstein polynomial.

\begin{theorem}
\label{bern_basis}
Let $M$ be a finitely generated $A_n$-module with system of generators
$\{f_1, \ldots, f_m \}$, 
and $E$ a free $A_n$-module with basis $e_1, \ldots, e_m$.
Let $\pi: E \rightarrow M$ be the natural $A_n$-epimorphism of $E$ onto $M$, defined by $\pi(e_i) = f_i$ for all $i= 1, \ldots, m$. Furthermore, let $N = Ker (\pi)$, and let $G = \{g_1, \ldots, g_s\}$ be a Gr\"{o}bner basis of $N$.
Finally, we treat $M$ as a filtered $A_n$-module with the filtration 
$M_r = \sum^m_{i =1} B_r f_i$, where $\{B_r \}_{r\in \nn}$ is the Bernstein filtration with $B_r = \{D \in A_n : deg(D) \le r \}$, and let 
$U_r$ denote the set $\{ w \in \Theta e : deg(w) \le r$, and $w$ is not a multiple of any $LM(g_j) \;(1 \le j \le s) \}$.
Then $\pi(U_r)$ is a basis of the $K$-vector space $M_r$. 
\end{theorem}

\begin{proof}
First we show that every element $\theta f_i$ ($1 \le i \le m, \; \theta \in \Theta$) not belonging to $\pi(U_r)$ can be written as a finite linear combination of elements of $\pi(U_r)$ with coefficients in $K$, so that the set $\pi(U_r)$ generates the $K$-vector space $M_r$. 
Since $\theta f_i \notin \pi(U_r)$, then $\theta e_i \notin U_r$, and thus $\theta e_i$ is a multiple of some $LM(g_j)$ ($1 \le j \le s$). Consider the element $g_j = a_j LM(g_j) + \cdots $ ($0 \neq a_j \in K$), where the dots stand for the sum of the other terms of $g_j$ whose monomials are less than $LM(g_j)$.
Since $g_j \in N$, we have  
$\pi(g_j) = a_j \pi(LM(g_j)) + \cdots  = 0$.
Since $\theta e_i$ is a multiple of $LM(g_j)$, that is, $\frac{\theta e_i}{LM(g_j)} = \theta'$, for some $\theta' \in \Theta$, we have 
$\pi(\theta' g_j) = a_j \pi(\theta' LM(g_j)) + \cdots = a_j \pi(\theta e_i) + \cdots  = a_j \theta f_i + \cdots = 0$. Thus, $\theta f_i$ is a finite $K$-linear combination of some elements $\tilde{\theta} f_k$ ($1 \le k \le m$) such that $\tilde{\theta} \in \Theta(r)$ and $\tilde{\theta} e_k < \theta e_i$. We can now apply induction on the well-ordered set $\Theta e$ to obtain that every element $\theta f_i$ ($\theta \in \Theta(r)$, $1 \le i \le m$) can be written as a finite linear combination of elements $\pi(U_r)$ with coefficients in the field $K$.

Next, we show that the set $\pi(U_r)$ is linearly independent over $K$. 
Suppose that $a_1 \pi(u_a) + \cdots + a_q \pi(u_q) = 0$ for some $u_1, \ldots , u_q \in U_r$ and $a_1, \ldots , a_q \in K$. Then $h = a_1 u_1+ \cdots +a_q u_q$ is an element of $N$. If an element $u = \theta e_j$ appears in $h$, that is, $u = u_i$ for some $i = 1, \ldots, q$, then $u$ is not a multiple of any $LM(g_i)$ (for  $1 \le i \le s$). Thus, $h$ is reduced with respect to $G$. By the first part of the proof of Theorem \ref{buch}, we have that $h =0$, and it follows that $a_1  = \cdots = a_q =0$. Therefore,  the set $\pi(U_r)$ is linearly independent over $K$, and this concludes the proof that $\pi(U_r)$ is a basis of $M_r$.
\end{proof}

The following theorem establishes the existence of the Bernstein polynomial and provides a formula for its computation. The constructive proof of existence below is the main result of this paper.

\begin{theorem}
\label{bern_exist}
Let $M$ be a finitely generated $A_n$-module with a system of generators $\{f_1, \ldots, f_m \}$, and let $\{M_r \}_{r \in \nn}$ be the corresponding filtration of $M$ with respect to the Bernstein filtration $\{B_r \}_{r\in \nn}$ $(M_r = \sum^p_{i =1} B_r f_i)$.
Furthermore,
Let $E$ be a free $A_n$-module with basis $e_1, \ldots, e_m$, let $N$ be the kernel of the natural epimorphism $\pi: E \rightarrow M$ $(e_i \mapsto f_i)$, and let $G =\{g_1, \ldots, g_s\}$ be a Gr\"{o}bner basis of $N$.
For any $i = 1, \ldots, m$, let $V_i = \{ (\alpha_1, \ldots, \alpha_n, \beta_1, \ldots, \beta_n) \in \nn^{2n}: x^{\alpha_1}\cdots x^{\alpha_n} \p^{\beta_1}\cdots \p^{\beta_n} e_i$ is not a multiple of any $LM(g_j)$ $(1 \le j \le s) \}$.
Then there exists a numerical polynomial $\chi (t)$ in one indeterminate $t$ such that:
\begin{enumerate}[{\em (i)}]
    \item 
    $\chi (r) = dim_K (M_r)$ for all sufficiently large $r \in \zz$ , that is, there exists $r_0 \in \zz$ such that the equality holds for all $r \ge r_0$.
    \item
    $deg(\chi (t)) \le 2n$, and the polynomial $\chi (t)$ can be represented as
\begin{align}
\label{bern_poly}
    \chi (t) =
                \sum^{2n}_{i =0} a_i \binom{t+i}{i}, 
\end{align}
where $a_i \in \zz$, and $i = 1, \ldots, 2n$.
    \item
\begin{align}
\label{compu_bern}
    \chi (t) = \sum_{i=1}^m \omega_i (t)
\end{align}
where $\omega_i(t) = \omega_{V_i}(t)$, and $\omega_{V_i}(t)$ is the Kolchin polynomial of the set $V_i \subseteq \nn^{2n}$ $(1 \le i \le n)$.
\end{enumerate}
\end{theorem}

\begin{proof}
With the notation of Theorem \ref{bern_basis},
the set $\pi(U_r)$ is a basis of the $K$-vector space $M_r$, for any $r \in \nn$. The second part of the proof of that theorem shows that the restriction of the map $\pi$ to $U_r$ is bijective because $Ker(\pi\vert_{U_r}) =0$. From this it follows that $dim_k M_r = Card~\pi(U_r) = Card~U_r$. 

Furthermore, for any $r \in \nn$,
\begin{align*}
    Card~U_r = \sum_{i=1}^m Card~V_i(r).
\end{align*}

By Proposition \ref{kolchin_dim}, for every $i= 1, \ldots , m$, there exists a numerical polynomial $\omega_i(t)$ 
in one indeterminate $t$, such that 
$\omega_i (r) =Card~V_i (r)$, for all sufficiently large $r \in \nn$. Hence, $Card~U_r = \sum_{i=1}^m \omega_i (r)$, for all sufficiently large $r$.
Proposition \ref{kolchin_dim} also shows that $deg( \sum_{i=1}^m \omega_i (t) ) \le 2n$, because each one of the terms in the sum has degree at most $2n$. 
Therefore, the polynomial $ \chi (t) = \sum_{i=1}^m \omega_i (t) $ satisfies all the conditions of the theorem.
\end{proof}

\begin{definition}
The polynomial $\chi(t)$, whose existence is established in Theorem \ref{bern_exist}, is called the \textbf{Bernstein polynomial} of the $A_n$-module associated with the system $\{f_1, \ldots, f_m \}$ of generators of $M$ (or with the filtration $\{M_r\}_{r \in \nn}$).
\end{definition}
\noindent
\textbf{Remark.} Formula (\ref{compu_bern}) reduces the computation of the Bernstein polynomial of an $A_n$-module to the computation of the Kolchin dimension polynomial of subsets of $\nn^m$.

\begin{example}
With the notation of Theorem \ref{bern_exist}, let $n =1$, and let an $A_1$-module $M$ be generated by a single element $f$ satisfying the defining equation $x^2 f + \p^2 f + x\p f = 0$. In other words, $M$ is a factor module of a free $A_1$-module $E = A_1 e$ with a free generator $e$ by its $A_1$-submodule $N = A_1 g$, where $g = x^2e + \p^2 e + x\p e$.  Clearly $\{g\}$ is a Gr\"{o}bner basis of $N$, from which we obtain the subset $\{(0,2)\}$ of $\nn^2$, corresponding to the multi-index of the leading monomial of $g$. Applying Proposition \ref{num_compu} and using the notation of Theorem \ref{bern_exist}, we obtain that $LM(g) = x^2$ and 
\begin{align*}
    \omega_{(2,0)}(t) =    2\binom{t+1}{1} - \binom{t}{0} 
    = 2t +1.
\end{align*}
Thus, the Bernstein polynomial of the module $M$ associated with the generator $f$ is $\chi (t) = \omega(t)_{(2,0)} = 2t +1$.
\end{example}

\begin{example}
Let $M$ be an $A_2$-module generated by the two elements $f_1$ and $f_2$ satisfying the defining equations
\begin{align*}
    (x^2_1 \p^3_1 + \p^5_1) f_1 =0, 
\quad \quad \text{and} \quad \quad
    x_2^2 f_1 - x_1 f_2 =0.
\end{align*}
Then $M$ is isomorphic to the factor module of a free $A_2$-module $E = A_2 e_1 + A_2 e_2$ with free generators $e_1$ and $e_2$ by its $A_2$-submodule $N$, where $N$ is as in Example \ref{grob_ex}. As shown in that example, the set
\begin{align*}
    G = \{
        g_1 &= x^2_1 \p^3_1 e_1 + \p^5_1 e_2,  \\
        g_2 &= x^2_2 e_1 - \p^5_1 e_2,  \\
        g_3 &= x_1^3\p_1^3 e_2 + 3x_1^2 \p_1^2 e_2 + x_1 \p_1^5 e_2 +5\p_1^4 e_2 \}
\end{align*}
is a Gr\"{o}bner basis of $N$. From this basis we get the subsets $V_1 =\{(2,0,3,0), (0,2,0,0)\}$ and $V_2 =\{(3,0,3,0)\}$ of $\nn^4$, corresponding to the multi-indices of the leading monomials of the elements of $G$ containing $e_1$ and $e_2$, respectively. For these subsets we can apply Proposition \ref{num_compu} to compute their Kolchin polynomials. Thus, using the notation of Theorem \ref{bern_exist}, we obtain
\begin{align*}
    &\omega_{V_1}(t) =
    \binom{t+4}{4} - \left[ \binom{t+4 -5}{4} + \binom{t+4 -2}{4}\right] + \binom{t+4 -7}{4}    \\
\text{and}  \\
   &\omega_{V_2}(t)    = \binom{t+4}{4} - \binom{t+4-6}{4}.
\end{align*}
Then the Bernstein polynomial of the module $M$ associated with the generator $f$ is given by $  \chi(t) = \omega_{V_1}(t) +  \omega_{V_2}(t)$. Therefore
\begin{align*}
   \chi(t) 
            = 6 \binom{t+3}{3} - 5\binom{t+2}{2} - 5 \binom{t+1}{1} + 15 \binom{t}{0}    
            = t^3 + \frac{7}{2} t^2 - \frac{3}{2} t  + 11.
\end{align*}
\end{example}

\begin{example}
Let $M$ be a $A_3$-module generated by one element $f$ satisfying the defining equations
\begin{align*}
(x_2 +1) f =0, \qquad x_1 f =0, \qquad \text{and} \qquad (\p_3 - 1) f =0.
\end{align*}
Then $M$ is isomorphic to the factor module of a free $A_3$-module $E = A_3 e$ with free generator $e$ by its $A_3$-submodule $N$ generated by
$g_1 = x_2 e + e, \,  
g_2 = x_1 e ,  \text{and} \,
g_3 = \p_3 e - e$.

As can be readily verified, $G = \{g_1, g_2, g_3 \}$ is a Gr\"{o}bner basis of $N$. The $S$-polynomials of $g_1$ and $g_2$ as well as $g_2$ and $g_3$ are given by $S(g_1, g_2) = S(g_2, g_3)= x_1 e $, which reduces to $0$ modulo $g_2$. The $S$-polynomials of $g_1$ and $g_3$ is given by $S(g_1, g_3) = x_2 e + \p_3 e $. It reduces modulo $g_1$ to $\p_3 e -e$, which in turn reduces to $0$ modulo $g_3$. Therefore, $G$ is indeed a Gr\"{o}bner basis of $N$. From this basis we get the subset $V = \{(0,1,0,0,0,0), (1,0,0,0,0,0), (0,0,0,0,0,1) \}$ of $\nn^6$, corresponding to the multi-indices of the leading monomials of the elements of $G$. Then applying Proposition \ref{num_compu} and using the notation of Theorem \ref{bern_exist}, we obtain
\begin{align*}
    \omega_V (t) &=
    \binom{t +6}{6} - 3\binom{t +6 -1}{6} + 3\binom{t +6 -2}{6} - \binom{t +6 -3}{6}  
    = \binom{t+3}{3}           \\
    &= \frac{1}{6}t^3 + t^2 +\frac{11}{6}t + 1.
\end{align*}

Therefore, the Bernstein dimension polynomial of the module $M$ associated with the generator $f$ is given by $\chi_M (t) = \omega_{V}(t)$.
\end{example}

The following result establishes the invariants of the Bernstein polynomial. 

\begin{theorem}
\label{invar}
Let $M$ be a finitely generated $A_n$-module and let
\begin{align*}
    \chi(t) =
                \sum^{2n}_{i =0} a_i \binom{t+i}{i}
\end{align*}
be the Bernstein polynomial associated with some finite system of generators $\{f_1, \ldots, f_m \}$ of $M$ (we write $\chi (t)$ in the form (\ref{int_canon}) with integer coefficients $a_i,~ 0 \le i \le 2n$).
Then the numbers $a_{2n}$, $d = deg(\chi (t))$ and  $a_d$ do not depend on the finite system of generators of the $A_n$-module $M$ with which this polynomial is associated. 
\end{theorem}

\begin{proof}
Let $\{g_1, \ldots, g_p \}$ be another finite system of generators of the $A_n$-module $M$, and let $\{M_r \}_{r \in \nn}$ and $\{M'_r \}_{r \in \nn}$ be the filtrations (with respect to the Bernstein filtration) associated with the systems of generators $\{f_1, \ldots, f_m \}$ and  $\{g_1, \ldots, g_p \}$, respectively. That is, for $r \in \nn$ we have
\begin{align*}
  M_r =  \sum^m_{i =1} B_r f_i,  
\quad \quad \text{and} \quad \quad
  M'_r =  \sum^p_{i =1} B_r g_i,
\end{align*}
where $\{B_r \}_{r \in \nn}$ is the Bernstein filtration with $B_r = \{D \in A_n : deg(D) \le r \}$.
Furthermore, let
\begin{align*}
    \chi^*(t) =
                \sum^{2n}_{i =1} b_i \binom{t+i}{i}
\end{align*}
be the Bernstein polynomial associated with the system $\{g_1, \ldots, g_p \}$.

Since $\bigcup_{r \in \nn} M_r = \bigcup_{r \in \nn} M'_r = M$ and all $M_r$ and $M'_r$ are finitely generated $K$-linear vector spaces, there exists an element $r_0 \in \nn$ such that 
$M_r \subseteq M'_{r+r_0}$ and
$M'_r \subseteq M_{r+r_0}$
for all $r \in \nn$.
It follows that
$\chi_M (r) \le   \chi^* (r+r_0)$ and $\chi^* (r) \le   \chi (r+r_0)$ for all sufficiently large $r \in \nn$. 
Therefore, $deg(\chi) = deg (\chi^{*}) = d$ and 
\begin{align*}
    a_d &= d! \lim_{r\to\infty} \frac{\chi  (r)}{r^d}  \\
     &\le d! \lim_{r\to\infty} \frac{\chi^{*}  (r +r_0)}{r^d}   \\
     &=  d! \lim_{r\to\infty} \frac{\chi^{*}  (r)}{r^d} 
     = b_d.
\end{align*}
Similarly, we have $b_d \le a_d$, and therefore $a_d = b_d$. Note that if $d < 2n$, then $a_{2n} =0$. This completes the proof that $d$ and $a_d$ are independent of the choice of generators of $M$.
\end{proof}

\textbf{Remark.} C. Dönch and A. Levin \cite[Theorem~5.9]{levin_13} considered bifiltrations of finitely generated $D$-modules and bivariate numerical polynomials associated with such filtrations. It turned out that these bivariate polynomials carry some invariants of the corresponding $D$-modules that are not carried by univariate Bernstein polynomials.

The degree of the polynomial $\chi (t)$ is called the \textbf{Bernstein dimension of $M$} and is denoted by $d(M)$. Since  $\chi (t)$ is a numerical polynomial, the number $d! a_d$ (where $d= deg(\chi(t))$ and $a_d$ is the coefficient of $\binom{t+d}{d}$ in (\ref{bern_poly})) is a positive integer called the \textbf{multiplicity of $M$} and is denoted by $m(M)$. 

Theorem \ref{invar}, in the next section, shows that both the dimension and multiplicity are invariants of the polynomial $\chi(t)$, that is, they do not depend on the choice of the system of generators $\{f_1, \ldots, f_m \}$ of $M$ with which the polynomial $\chi(t)$ is associated.

\section{Generalized Krull Dimension of \texorpdfstring{$D$}{}-modules}
\label{sec:6}
 \setcounter{equation}{0}
 
As an application of the properties of the Bernstein polynomial, we now turn to a generalization of the concept of the Krull dimension of a commutative ring to the case of modules over the Weyl algebra $A_n$. Generalizations for the case of differential and difference modules similar to the one presented here are developed with corresponding results in \cite{john} and \cite[Section~3.6]{levin_08}, respectively. We first state some definitions. 

Let $A_n$ be a Weyl algebra, $M$ a finitely generated  $A_n$-module, and $U$ a family of $A_n$-submodules of $M$. 
Furthermore, let $\mathscr{B}_U$ denote the set of all pairs $(N,N') \in U \times U$ such that $N \supseteq N'$, and let $\Bar{\zz}$ be the augmented set of integers, that is, the set obtained by adjoining a new symbol, $\infty$, to the set $\zz$ ($\Bar{\zz}$ is considered as a totally ordered set whose order $>$ is the extension of the natural order of $\zz$ such that $a < \infty$ for any $a \in \zz$). With this notation and conventions, we have the following proposition. 

\begin{proposition}
There exists a unique mapping 
$\mu_U :
\mathscr{B}_U  \rightarrow  \Bar{\zz}$,
with the following properties:
\begin{enumerate} [{\em (i)}]
    \item $\mu_U (N,N') \geq -1$ for every pair $(N,N') \in \mathscr{B}_U$;
    \item if $d \in \nn$, then  $\mu_U (N,N') \geq d$ if and only if $N \neq N'$ and there exists and infinite chain
    $N \supseteq N_0  \supseteq N_1 \supseteq \cdots  \supseteq N'$
    such that  $\mu_U (N_{i-1},N_{i}) \geq d-1$, for all $i = 1, 2 \ldots $.
\end{enumerate}
\end{proposition}
\begin{proof}
We define the desired mapping $\mu_U$ as follows.
If $N \in U$, we set $\mu_U (N,N) = -1$. If $(N,N') \in \mathscr{B}_U$, $N \neq N'$, and condition (ii) holds for all $d \in \nn$, we set $\mu_U (N,N') = \infty$. Finally, if a pair $(N,N') \in \mathscr{B}_U$ satisfies condition (ii) for some $d \in \nn$, we define $\mu_U(N,N')$ as the greatest of such integers $d$. It is easy to see that the mapping $\mu$ is well-defined, satisfies both conditions (i) and (ii), and is uniquely determined by these conditions.
\end{proof}

With the notation above we state the following definitions.

\begin{definition}
The least upper bound of the set 
$\{ \mu_U (N,N') : (N,N')  \in \mathscr{B}_U \}$ 
is called the \textbf{type} of the $A_n(K)$-module $M$ over the family of its $A_n$-submodules $U$, and is denoted by $type_U (M)$.
\end{definition}
\begin{definition}
Let $A_n$ be a Weyl algebra, $M$ a finitely generated $A_n$-module, and $U$ a family of $A_n$-submodules of $M$. 
Then the least upper bound of the lengths of all chains
$N_0 \supsetneq N_1 \supsetneq \cdots \supsetneq N_p$,
where $N_i \in U$ and $\mu_U (N_{i-1},N_{i}) = type_U M \; (i= 1, \ldots , p$ and the number $p$ is considered as the length of the chain$)$, is called the \textbf{dimension} of the $A_n$-module $M$ over the family $U$.
\end{definition}
The dimension of an $A_n$-module $M$ over a family of its $A_n$-submodules $U$ is denoted by $dim_U M$.

Let $M$ be a $A_n$-module and let $U$ be a family of $A_n$-submodules of $M$. It is easy to see that if $type_U M < \infty$, then $dim_U M \ge 1$. At the same time, if $type_U M = \infty$, then $dim_U M$ may equal zero. 

%
%
%
%

In what follows we use the following notation, where $m(M)$ denotes the multiplicity of $M$, defined in Section 4:
\begin{align*}
    \delta (M) =
\begin{cases}
0 \quad  &\text{if}  \quad d < 2n ,   \\
m(M) \quad  &\text{if}  \quad d  = 2n .
\end{cases}
\end{align*}

The following proposition, proved in \cite[Chapter~7,~\S~5]{cou_95}, shows that a good filtration $\Gamma$ of an $A_n$-module $M$ induces good filtrations on $A_n$-submodules and factor modules of $M$. These filtrations are called the filtrations \textit{induced} by $\Gamma$.

\begin{proposition}
\label{indu_fil}
Let $M$ be an $A_n$-module with the filtration $\Gamma = \{M_r\}_{r \in \nn}$ with respect to the Bernstein filtration $\{B_r \}_{r \in \nn}$, and let $N$ be a $A_n$-submodule of $M$. Then $\{N \cap M_r \}_{r \in \nn}$ and $\{M_r /(M_r \cap N) \}_{r \in \nn}$ are good filtrations for $N$ and $M /N$, respectively.
\end{proposition}

We will use the following lemma and proposition to arrive at the central result of this section, Theorem \ref{Krull-type}, about the generalized Krull dimension and type of $D$-modules. First, we recall a definition from homological algebra \cite{atiyah_mac}.

Let $R$ be a ring. A sequence of $R$-modules and $R$-homomorphisms
\begin{align*}
    \cdots \longrightarrow M_{i-1} \xrightarrow{\; \phi_i \;} M_i \xrightarrow{\phi_{i+1}} M_{i+1} \longrightarrow \cdots
\end{align*}
is said to be \textit{exact} at $M_i$ if $Im (\phi_i) = Ker (\phi_{i+1})$. The sequence is \textit{exact} if it is exact at each $M_i$.

\begin{lemma}
\label{exact_dim}
Let $K$ be a field, and let $0 \rightarrow N  \xrightarrow{i} M  \xrightarrow{j} P  \rightarrow 0 $ be a an exact sequence of finitely generated $A_n$-modules. Then $\delta(N) + \delta(P) = \delta (M)$. 
\end{lemma}

\begin{proof}
Let $\{M_r\}_{r \in \nn}$ be a good filtration of the $A_n$-modules $M$ and set $N_r = i^{-1} (i(N) \cap M_r)$ and $P_r = j(M_r)$ for any $r \in \nn$. Clearly, the filtration $\{P_r \}_{r \in \nn}$ of $P$ is good, and by Proposition \ref{indu_fil} the filtration $\{N_r\}_{r \in \nn}$ of $N$ is also good. Let $\chi_N (t)$, $\chi_M (t)$ and $\chi_P (t)$ be the Bernstein dimension polynomials of $N$, $M$ and $P$, respectively, associated with these filtrations. Since for all sufficiently large $r \in \nn$ the sequence 
\begin{align*}
    0 \rightarrow N_r \rightarrow M_r \rightarrow j(M_r) \rightarrow 0
\end{align*}
is exact, it follows that $dim_K N_r + dim_K \; j(M_r) = dim_K M_r$, and so $\chi_N (t) + \chi_P (t) = \chi_M (t)$. Therefore, by the last statement in Proposition \ref{canon}, we have
\begin{align*}
    \delta(M) &= \Delta^{2n} (\chi_M (t)) = \Delta^{2n} \big(\chi_N (t) + \chi_P (t) \big)  \\
    &= \Delta^{2n} (\chi_N (t) ) + \Delta^{2n} (\chi_P (t) )
    = \delta(N) + \delta (P),
\end{align*}
thus establishing the result.
\end{proof}

\begin{proposition}
\label{min_indep}
Let $M$ be a finitely generated $A_n$-module. Then the number $\delta(M)$ is equal to the maximal number of elements of the module $M$ that are linearly independent over $A_n$.
\end{proposition}

\begin{proof}
First, we will show that $\delta (M) =0$ if and only if every element of $M$ is linearly dependent over $A_n$. Suppose that $\delta(M)= 0$ and, for the sake of contradiction, that there exists $z \in M$ that is linearly independent over $A_n$. Then the map $\phi : A_n \rightarrow M$, defined by $\phi(D)= Dz$ for every $D \in A_n$, is a monomorphism of $A_n$-modules. Applying Lemma \ref{exact_dim} to the exact sequence of finitely generated $A_n$-modules $0 \rightarrow A_n \xrightarrow{\phi} M \rightarrow M/ \phi(A_n) \rightarrow 0$, we have
$\delta(A_n) = \delta(M) - \delta(M/ \phi(A_n)) \le \delta(M) = 0$. But $\delta(A_n) = m(A_n) =1$, because $d(A_n) =2n$. This contradiction proves that any element of $M$ is linearly dependent over $A_n$. 

Conversely, suppose that every element of $M$ is linearly dependent over $A_n$. Let $\xi_1, \ldots, \xi_k$ be generators of the $A_n$-modules $M$ (so that $M = \sum^k_{i =0}A_n \xi_i)$. Let $N_i = A_n \xi_i~ (1 \le i \le k)$, and $L_i$ denote the kernel of the mapping $\psi_i: A_n \rightarrow N_i~(1 \le i \le k)$ such that $\psi_i(u) = u \xi_i$ for any $u \in A_n$. By assumption, $\xi_i$ is linearly dependent over $A_n$, and therefore, $L_i \neq 0~(1 \le i \le k)$. We have $\delta(L_i) \neq 0$ (because $\delta(L_i) = 0$ implies that all elements of $A_n$ are linearly dependent, giving $\xi_i u \neq 0~(0 \neq u \in L_i)$, and $A_n$ would have zero divisors, which is not the case). Applying Lemma \ref{exact_dim} to the exact sequence of finitely generated $A_n$-modules $ 0 \rightarrow L_i \rightarrow A_n \xrightarrow{\psi_i} N_i \rightarrow 0$, we obtain that $\delta(L_i) + \delta(N_i)= \delta(A_n) =1$. Since each nonzero element of $L_i$ is linearly independent over $A_n$, the above reasoning shows that $\delta(L_i) \neq 0$, and hence $\delta(L_i) \ge 1~ (1 \le i \le k)$. Therefore, $0 \le \delta(N_i) = \delta(A_n) - \delta(L_i) = 1 - \delta(L_i) \le 0$, and so $\delta(N_i) = 0~(1 \le i \le k)$. Since $M = \sum^k_{i =1}N_i$, we have $0 \le \delta (M) = \delta \left( \sum^k_{i =1}N_i \right) \le \sum^k_{i =1} \delta(N_i) =0$. To prove the last inequality, it is sufficient to show that $\delta(N_1 +  N_2) \le \delta(N_1) + \delta(N_2)$ for all finitely generated $A_n$-modules $N_1$ and $N_2$. Let $M' = N_1 + N_2$. Applying Lemma \ref{exact_dim}  to the canonical exact sequence $0 \rightarrow N_1 \rightarrow M' \rightarrow M'/N_1 \rightarrow 0$ we obtain that $\delta(M') = \delta(N_1) + \delta(M'/N_1) = \delta(N_1) + \delta((N_1 +N_2)/N_1) + \delta(N_1) + \delta(N_2/(N_1 \cap N_2)) \le \delta(N_1) + \delta(N_2)$. Thus, it has been shown that every element of $M$ is linearly independent over $A_n$ if and only if $\delta(M) = 0$.

Now, to complete the proof of the proposition, let $p$ be the maximal number of elements of the $A_n$-module $M$ that are linearly independent over $A_n$. Let $\{z_1, \ldots, z_p\}$ be any system of elements of $M$ that are linearly independent over $A_n$, and let $F = \sum^p_{i=1} A_n z_i$. Then $\{ \sum^p_{i=1} A_{n_r} z_i \}_{r \in \nn}$ is a good filtration of $F$. By Proposition \ref{kolchin_dim} (ii), the Bernstein polynomial associated with this filtration has the form $\chi (t) = p \binom{t+2n}{2n}$, and so $\delta(F) = \Delta^{2n} \chi(t) = p$. Furthermore, the maximality of the system $\{z_1, \ldots, z_p\}$, linearly independent over $A_n$, implies that every element of the finitely generated module $M/F$ is lineraly dependent over $A_n$, and hence $\delta(M/F) = 0$. Applying Lemma  \ref{exact_dim} to the exact sequence of finitely generated $A_n$-modules  
\begin{align*}
    0 \rightarrow F \rightarrow M \rightarrow M/F \rightarrow 0,
\end{align*}
we obtain that $\delta(M) = \delta(F) + \delta(M/F) = \delta(F) =p$.
\end{proof}

\begin{theorem}
\label{Krull-type}
Let $M$ be a finitely generated $A_n$-module, and let $U$ be the family of all $A_n$-submodules of $M$.  Then we have the following:
\begin{enumerate}[{\em (i)}]
    \item If $\delta(M) > 0$, then $type_U M = 2n$ 
    and 
    $dim_U M = \delta(M)$;
    \item If $\delta(M) = 0$, then $type_U M < 2n$ .
\end{enumerate}
\end{theorem}
\begin{proof}
Let $\{M_r\}_{r \in \nn}$ be a finite system of generators of the $A_n$-module $M$, and let
$M_r = \sum^m_{i =1} B_r f_i$ be the corresponding filtration of $M$ with respect to the Bernstein filtration $\{B_r\}_{r \in \nn}$, where $B_r = \{D \in A_n : deg(D) \le r \}$. Let
$N,L \in U$, with $N \supseteq L$. 
By Proposition \ref{indu_fil}, $\{N \cap M_r\}_{r\in \nn}$ and $\{L \cap M_r\}_{r \in \nn}$ are good filtrations of the $A_n$-submodules $N$ and $L$, respectively. 
Thus, we can consider the corresponding Bernstein polynomials
$\chi_N(t)$ and $\chi_L(t)$ associated with these filtrations. These polynomials describe the dimension of the vector space $\{N \cap M_r\}_{r \in \nn}$ over the field $K$, and the dimension of the vector space $\{L \cap M_r\}_{r \in \nn}$ over the field $K$, respectively. The inclusion $N \supseteq L$, implies 
$N \cap M_r \supseteq L \cap M_r$ for any $r \in \nn$, and
thus 
$\chi_L(t) \le \chi_N(t)$ (where $ \le$ denotes the natural order on the set of numerical polynomials in one indeterminate $t$: $f(t) \le g(t)$ if and only if $f(r) \le g(r)$ for all sufficiently large $r \in \nn$).
Furthermore, $N = L$ if and only if $\chi_N(t) = \chi_L(t)$ (indeed, if $\chi_N(t) = \chi_L(t)$ but $N \subsetneq L$, then there exists $r_0 \in \nn$ such that $(N \cap M_r) \subsetneq (L \cap M_r)$ for all $r > r_0$, in which case $\chi_N(t) < \chi_L(t)$, contradicting our assumption). 

Next, we will prove that if  $(N,L) \in \mathscr{B}_U = \{(P,Q) \in U \times U : P \supseteq Q \}$ and
$\mu_U (N,L) \ge d$ ($d \in \zz, d \ge -1$), then $deg(\chi_N(t) - \chi_L(t)) \ge d$.
We proceed by induction on $d$. Since $deg(\chi_N(t) - \chi_L(t)) \ge -1$ for any pair $(P,Q) \in \mathscr{B}_U$, and $deg(\chi_N(t) - \chi_L(t)) \ge 0$ if $L \supsetneq N$ (as we noted above, in this case $\chi_L(t) < \chi_N(t)$), then the statement is true for $d=-1$ and $d=0$.

Now, let $d > 0$, and for any $i \in \nn$, with $i <d$, suppose the inequality $\mu_U (N,L) \ge i $ $((N, L) \in \mathscr{B}_U$) implies that $deg(\chi_N(t) - \chi_L(t)) \ge i$. Assume further that $(N, L) \in \mathscr{B}_U$ and $\mu_U (N,L) \ge d$ holds. Then there exists an infinite strictly descending chain  
\begin{align*}
    N = N_0 \supsetneq N_1 \supsetneq \cdots \supsetneq L
\end{align*}
of $A_n$-submodules of $M$ such that $\mu_U (N_{i-1},N_i) \ge d-1$ for $i = 1, 2, \ldots$. If $deg(\chi_{N_{i-1}}(t) - \chi_{N_i}(t)) \ge d$ for some $i \ge 1$, then $deg(\chi_N(t) - \chi_L(t)) \ge d$ and the statement would be proven.

For the sake of contradiction, suppose that $deg(\chi_{N_{i-1}}(t) - \chi_{N_i}(t)) = d -1$ for all $i = 1, 2, \ldots$. Then every polynomial $\chi_{N_{i-1}}(t) - \chi_{N_i}(t)$ ($i \in \nn, i \ge 1$) can be written in the form
\begin{align*}
    \chi_{N_{i-1}}(t) - \chi_{N_i}(t)
=
    \sum^{d-1}_{k=0} c_{i,k} \binom{t+k}{k},
\end{align*}
where $c_{i,0}, \ldots, c_{i,d-1} \in \zz$ and $c_{i, d-1} >0$
(see Proposition \ref{canon}).
In this case, 
$\chi_N(t) - \chi_{N_i}(t) = \sum^{d-1}_{k=0} c'_{ik}\binom{t+k}{k}$, 
where $c'_{i,0}, \ldots, c'_{i,d-1} \in \zz$ and 
$c'_{1 ,d-1} = \sum^i_{\nu =1} c_{\nu ,d-1}$.
Therefore, $c'_{1, d-1} < c'_{2 ,d-1} < \cdots $. If 
$deg(\chi_N(t) - \chi_L(t)) = d -1$, then $\chi_N(t) - \chi_L(t) = \sum^{d-1}_{k=0} a_k \binom{t+k}{k}$. Since  
$\chi_{N_{i-1}}(r) - \chi_{N_i}(r) < \chi_N(r) - \chi_L(r)$ for all sufficiently large $r$, then $c'_{1, d-1} < c'_{2 ,d-1} < \cdots  < a_{d-1}$, a contradiction because $a_{d-1}$ is a fixed number and thus it is impossible to have a strictly increasing chain of natural numbers less than $a_{d-1}$. Therefore,
$deg(\chi_N(t) - \chi_L(t)) \ge d$ for any pair $(N, L) \in \mathscr{B}_U$ with $\mu_U (N,L) \ge d$.

This last statement, together with the fact that $deg(\chi_N(t) - \chi_L(t)) \le 2n$ for any $(N, L) \in \mathscr{B}_U$, implies that $\mu_U (N,L) \le 2n$ for all pairs $(N,L) \in \mathscr{B}_U$. It follows that
\begin{align}
\label{type}
    type_U M \le 2n.
\end{align}

Next we will prove that $type_U M \ge 2n$. If $\delta(M) = p > 0$, then Proposition \ref{min_indep} shows that $M$ contains $p$
elements $z_1, \ldots, z_p$ that are linearly independent over $A_n$. If $U'$ is the family of all $A_n$-submodules of the $A_n$-module $A_n z_1$, then $type_{U'} A_n z_1 \le type_U M$. Thus, to complete the proof of the first statement of the theorem, it is sufficient to consider an $A_n$-module $M =A_nz_1$, where $z_1$ is linearly independent over $A_n$, and show that $type_{U'} A_n z_1 \ge 2n$.
Let $U''$ denote the family of all $A_n$-submodules of $A_n z_1$ that can be represented as finite sums of $A_n$-modules of the form $A_n x^\alpha \p^\beta z_1$, where $\alpha$ and $\beta$
are multi-indices from $\nn^n$. We are going to use induction on $n$ to show that  $\mu_{U''} (A_n z_1, 0) \ge 2n$. If $n =0$, the inequality is obvious. So let $n >0$. Since the element $z_1$ is linearly independent over $A_n$, we have the following strictly descending chain
\begin{align*}
    A_n z_1 \supsetneq A_n x_n z_1 \supsetneq A_n x^2_n z_1 \supsetneq \cdots \supsetneq 0
\end{align*}
of elements of $U''$. Let us show that $\mu_{U''}(A_n x^{i-1}_n z_1, A_n x^i_n z_1) \ge 2n -1$ for all $i = 1, 2, \ldots$. Let $L = A_n x^{i-1}_n z_1 / A_n x^i_n z_1 $ and let $y$ be the image element $ x^{i-1}_n z_1 $ under the natural epimorphism $ A_n x^{i-1}_n z_1 \rightarrow L$.
Then $x_n y =0$ and $x_n \theta' y = \theta' x_n y =0$ for any element $\theta' \in \Theta'$, where 
\begin{align*}
    \Theta' 
    =
    \{ \prod^n_{i=1} \prod^n_{j=1} x^{\alpha_i}_i \p^{\beta_j}_j 
    : \alpha_1, \ldots, \alpha_n, \p_1, \ldots, \p_n \in \nn \}.
\end{align*}

Let $A'_n = A_{n-1}\p_n$, where $A_{n-1}$ is the Weyl algebra generated by $x_1, \ldots, x_{n-1}$ and $\p_1, \ldots, \p_{n-1}$.
Then $L = A_n' y$, and the element $y$ is linearly independent over $A_n'$.

Furthermore, $\mu_{U''} (A_n x^{i-1}_n z_1, A_n x^i_n z_1) = \mu_{U''_1} (L,0)$ where $U''_1$ is the family of all $A_n'$-submodules of $L$ that can be represented as finite sums of the $A_n'$-modules of the form $A_n'\theta' y~(\theta' \in \Theta)$.
The chain $L \supseteq A_{n-1} \p_n y \supseteq A_{n-1} \p^2_n y \supseteq \cdots \supseteq \{0\}$
shows that if $U^*$ is the set of all $A_{n-1}$-submodules of $C_i = A_n \p^i_n y / A_{n-1} \p^{i+1}_n y~ (i \in \nn)$, then 
$\mu_U (A_n z_1, 0) \ge \mu_{U^*} (C_i, 0)+2 \ge 2(n-1)$. 
By the induction hypothesis,
$\mu_{U''_1} (L, 0) \ge 2n-1$, and so $\mu_{U''} (A_n x^{i-1}_n z_1, A_n x^i_n z_1) \ge 2n -1~ (i= 1, 2, \ldots )$. Thus $\mu_{U''} (A_n z_1, 0) \ge 2n$.

Hence, 
\begin{align*}
    type_U M \ge type_{U'} A_n z_1 \ge type_{U''} A_n z_1 \ge 2n.
\end{align*}
Combining these inequalities with (\ref{type}), we obtain
\begin{align*}
    type_U M = 2n.
\end{align*}

The proof of this last equality also shows that $type_{U'_i} A_n z_i = 2n$  for every  $i= 1, 2, \ldots$, where $U'_i$ denotes the family o all $A_n$-submodules of the $A_n$-module $A_n z_i$. Therefore, \newline
$\mu_U \left( \sum^k_{i=1} A_n z_i, \sum^{k-1}_{i=1} A_n z_i\right) = 2n = type_U M$ for every $k= 1, \ldots, q $, and the chain
\begin{align*}
    M =
    \sum^q_{i=1} A_n z_i \supsetneq
    \sum^{q-1}_{i=1} A_n z_i  \supsetneq
    \cdots
    \supsetneq
    A_n z_1  \supsetneq
    0
\end{align*}
leads to the inequality
\begin{align}
\label{delta_in}
    dim_U M \ge \delta(M).
\end{align}

Let $N_0 \supsetneq N_1 \supsetneq \cdots \supseteq N_p$ be a chain of $A_n$-submodules of $M$ such that $\mu_U (N_{i-1}, N_i) = type_U M =2n$, for $i= 1, \ldots, p$. The arguments used at the beginning of the proof show that the degree of each polynomial $\chi_{N_{i-1}} (t) - \chi_{N_i} (t)~ (1 \le i \le p)$ is equal to $2n$, and so every such polynomial can be written as
\begin{align*}
    \chi_{N_{i-1}} (t) - \chi_{N_i} (t)
    =
     \sum^{2n}_{k=0} a_{ik} \binom{t+k}{k},
\end{align*}
where $a_{ik} \in \zz~ (0 \le k \le2 n)$. It follows that
\begin{align*}
    \chi_{N_0} (t) - \chi_{N_p} (t)
    &= 
    \sum^p_{i=1} \big( \chi_{N_{i-1}} (t) - \chi_{N_i} (t) \big)   \\
    &=  
      \sum^{2n}_{k=0} a'_k \binom{t+k}{k},
\end{align*}
where $a'_i \in \zz~ (0 \le k \le 2n)$.
\medskip

On the other hand,
\begin{align*}
    \chi_{N_0} (t) - \chi_{N_p} (t)
    \le
    \chi_{N_0} (t) 
    \le 
    \chi_M (t) 
    =
    \sum^{2n}_{k=0} a_k \binom{t+k}{k},
\end{align*}
where $a_{2n} = \delta (M)$ (we write $f(t) \le g(t)$ for two numerical polynomials $f(t)$ and $g(t)$ if $f(r) \le g(r)$ for all sufficiently large $r \in \nn$). It follows that $a_{2n} \ge a'_{2n} \ge p$, and hence $\delta(M) \ge dim_U M$. Combining the last inequality with (\ref{delta_in}) we obtain the desired inequality $dim_U M = \delta(M)$.

Finally, we prove Part (ii) of the theorem. If $\delta (M) =0$, then for any pair $(N, L) \in \mathscr{B}_U$, we have the equality $\delta (N) = \delta (L) = 0$, and hence $deg(\chi_N (t)) < 2n, deg(\chi_L (t)) <2n$, and $deg\big( \chi_N (t) - \chi_L (t) \big) < 2n$. As it has been shown above, the last inequality implies the inequality $\mu_U (N, L) < 2n$. Therefore, if $\delta (M) =0$, then $type_U M  < 2n$.
\end{proof}

\section*{Acknowledgments}

The author would like to thank Dr. Alexander Levin who supervised the master’s thesis research \cite{priet} that contained an early version of these results.

%

%
%
%
%

%

\vskip0.2in
\begin{minipage}[b]{9 cm}
      Email: {\it 22prieto@cua.edu}
\end{minipage}

\end{document}